\documentclass{amsart}
\usepackage{graphicx}
\newtheorem{thm}{Theorem}[section]

\newtheorem{lem}[thm]{Lemma}
\newtheorem{prop}[thm]{Proposition}
\theoremstyle{definition}
\newtheorem{defn}[thm]{Definition}
\theoremstyle{remark}
\newtheorem{rem}[thm]{Remark}
\newtheorem{example}[thm]{Example}
\numberwithin{equation}{section}

\begin{document}

\title[Conservative and semiconservative random walks]
{Conservative and semiconservative random walks:\\
 recurrence and transience}%
\author{Vyacheslav M. Abramov}%
\address{24 Sagan Drive, Cranbourne North, Victoria-3977, Australia}%
\email{vabramov126@gmail.com}%

\subjclass{60G50; 60J80; 60C05; 60K25}%
\keywords{Multi-dimensional random walks; birth-and-death process;
Markovian single-server queues}%

\begin{abstract}
In the present paper we define conservative and semiconservative
random walks in $\mathbb{Z}^d$ and study different families of
random walks. The family of symmetric random walks is one of the
families of conservative random walks, and simple (P\'olya) random
walks are their representatives. The classification of random walks
given in the present paper enables us to provide a new approach to
random walks in $\mathbb{Z}^d$ by reduction to birth-and-death
processes. We construct nontrivial examples of recurrent random
walks in $\mathbb{Z}^d$ for any $d\geq3$ and transient random walks
in $\mathbb{Z}^2$.
\end{abstract}

\maketitle
\section{Introduction}

It is well-known that the simple random walk (which is also called
P\'olya random walk) of dimension one or two is recurrent, i.e.
being started from the origin, it returns to the original point
infinitely many times. P\'olya random walks of dimension higher than
two are transient. These facts were originally established in 1921
by P\'olya \cite{Polya}, and nowadays there are a number of new
proofs and extensions of the P\'olya theorem (e.g.
\cite{ChungFuchs}, \cite{DoyleSnell}, \cite{FosterGood},
\cite{LyonsPeres}, \cite{Lyons}, \cite{Novak}, and others).

Although the aforementioned results are well-known and have the
distinguished history, the fact that random walks of dimension
higher than two are not recurrent, while those of dimensions one and
two are, remains mysterious and lying beyond the intuitive
understanding. The further extensions of the P\'olya theorem are
chiefly based on application of the majorization (coupling)
techniques of electric networks (e.g. \cite{LyonsPeres}) or general
methods using probability inequalities (e.g. \cite{ChungOrnstein} or
\cite{Durett}, Lemma 4.2.5). They extend the result for the random
walks of the same dimension to address the question whether or not a
modified random walk of dimension two is recurrent or that of
dimension three or higher is transient. As well, broad extensions to
P\'olya's theorem may be obtained by the method of Lyapunov
functions, dating back to Lamperti \cite{Lamperti}.

We suggest another approach for classification of random walks. The
main idea of our approach is to establish connection between random
walks and birth-and-death processes. On the basis of this
connection, we define new classes of random walks, called
\textit{conservative} and \textit{semiconservative}. With the aid of
this classification we are able to study the new cases, that was
impossible in the frameworks of the earlier methods.

Our attention in the present paper is restricted by the random walks
$\mathbf{S}_t=\big(S_t^{(1)},S_t^{(2)},\ldots,S_t^{(d)}\big)$ in
$\mathbb{Z}^d$ ($t=0,1,\ldots$ is a discrete time parameter)
satisfying the recurrence relations
\begin{eqnarray}
\mathbf{S}_0&=&\mathbf{0},\label{eq-1.1}\\
\mathbf{S}_{t}&=&\mathbf{S}_{t-1}+\mathbf{e}_t, \
t\geq1,\label{eq-1.2}
\end{eqnarray}
where $\mathbf{0}$ is the $d$-dimensional vector of zeros, and the
random vector $\mathbf{e}_t$ takes the values $\pm\mathbf{1}_i$,
$i=1,2,\ldots,d$, and $\mathbf{0}$.

Let $\mathcal{A}$ be a set of (vector-valued) parameters. Then, the
triple $\{\mathbf{S}_t,\mathcal{A},d\}$ is said to specify a family
of random walks. Let ${a}\in\mathcal{A}$. Then,
$\mathbf{S}_t({a,d})$ is a random walk that belongs to the family of
random walks $\{\mathbf{S}_t,\mathcal{A},d\}$.

\smallskip

The random walks that are studied in this paper are originated by
the following three models.

\smallskip

\textit{Model 1}. For random walks defined by \eqref{eq-1.1} and
\eqref{eq-1.2}, let the vector $\mathbf{e}_t$ be one of the $2d$
randomly chosen vectors $\{\pm \mathbf{1}_i, i=1,2,\ldots,d\}$
independently of the history and each other as follows. The
probability that the vector $\mathbf{1}_i$ will be chosen, which is
the same for the vector $(-\mathbf{1}_i)$, is equal to $\alpha_i>0$,
and $2\sum_{i=1}^d\alpha_i=1.$ Then the family
$\{\mathbf{S}_t,\mathcal{A},d\}$ is specified by the set of vectors
$\big(\alpha_1, \alpha_2,\ldots, \alpha_d\big)\in\mathcal{A}$.

The random walks of Model 1 are called \textit{symmetric} random
walks, while the title \textit{simple} or \textit{simple symmetric}
is related to P\'olya random walks, which is one of the
representatives of the symmetric random walks family.

\smallskip

\textit{Model 2}. Assume now, that the vector $\mathbf{e}_t$ is
specified as follows. If $S_{t-1}^{(i)}=0$, then the vectors
$\mathbf{1}_i$ and $(-\mathbf{1}_i)$ are chosen with equal
probability $\alpha_i$ each, where $2\sum_{i=1}^d\alpha_i=1$. If
$S_{t-1}^{(i)}\neq0$, then the probability to be chosen for each of
the vectors $\mathbf{1}_i$ and $(-\mathbf{1}_i)$ is
$\alpha_i-\delta_i/2>0$, where $\delta_i>0$. Then, the vector
$\mathbf{0}$ is chosen with the complementary probability. For
instance, if $\mathbf{S}_{t-1}=(1,3,-1)$, then the probability for
vector $\mathbf{0}$ to be chosen is $(\delta_1+\delta_2+\delta_3)$.
But if $\mathbf{S}_{t-1}=(0,2,0)$, then the probability for vector
$\mathbf{0}$ to be chosen is $\delta_2$. The family of random walks
in this model is specified by the sets of vectors $\big(\alpha_1,
\alpha_2,\ldots, \alpha_d\big)$ and $\big(\delta_1, \delta_2,\ldots,
\delta_d\big)$, where the components of these vectors satisfy the
inequality
\begin{equation*}\label{eq-1.7}
1 \  \leq\frac{2\alpha_i}{2\alpha_i-\delta_i}\ <\ \infty.
\end{equation*}

\smallskip

\textit{Model 3}. Assume now, that the vector $\mathbf{e}_t$ is
specified as follows. If $S_{t-1}^{(i)}\neq0$, then the vectors
$\mathbf{1}_i$ and $(-\mathbf{1}_i)$ are chosen with equal
probability $\alpha_i$ each, where $2\sum_{i=1}^d\alpha_i=1$. If
$S_{t-1}^{(i)}=0$, then the probability to be chosen for each of the
vectors $\mathbf{1}_i$ and $(-\mathbf{1}_i)$ is
$\alpha_i-\delta_i/2>0$, where $\delta_i>0$. Then, the vector
$\mathbf{0}$ is chosen with the complementary probability. As in
Model 2, the family of random walks is specified by the sets of
vectors $\big(\alpha_1, \alpha_2,\ldots, \alpha_d\big)$ and
$\big(\delta_1, \delta_2,\ldots, \delta_d\big)$, and the components
of these vectors satisfy the inequalities
\begin{equation*}\label{eq-1.8}
0 \ <  \ \frac{2\alpha_i-\delta_i}{2\alpha_i}\  \leq\ 1.
\end{equation*}

\smallskip
In the case of Models 2 and 3, the set $\mathcal{A}$ is the sets of
vectors $\big(\alpha_1, \alpha_2,\ldots, \alpha_d\big)$ and
$\big(\delta_1, \delta_2,\ldots, \delta_d\big)$. That is, an element
$a\in\mathcal{A}$ is the vector of dimension $2d$. It is easy to see
that Model 1 is particular to both of Models 2 and 3 when
$\delta_i=0$, $i=1,2,\ldots,d$. All the three models could be
amalgamated to a common model, the study of which can be further
extended. For the purpose of the present paper, however, it is
profitable to start from Models 1, 2 and 3 in order to initiate the
study of more general models in a gradual way.

Let $\mathbf{n}=\big(n^{(1)}, n^{(2)},\ldots, n^{(d)}\big)$ be a
vector in $\mathbb{Z}^d$. Its norm is defined by
$$
\|\mathbf{n}\|=\sum_{i=1}^d\big|n^{(i)}\big|.
$$

By the sequence of \textit{active} time instants $t_1, t_2,\ldots,
t_j,\ldots$, we mean
$$
\begin{aligned}
t_1=&\inf\{t>0: \|\mathbf{S}_t\|>0\}, \quad t_2=\inf\{t>t_1:
\|\mathbf{S}_t\|\neq\|\mathbf{S}_{t_1}\|\},\ldots,\\
&t_{j}=\inf\{t>t_{j-1}:
\|\mathbf{S}_t\|\neq\|\mathbf{S}_{t_{j-1}}\|\},\ldots
\end{aligned}
$$
In other words, the active time instants are the times when a random
walk changes its state. Note, that for the family of random walks in
Model 1, all the time instants $t=1,2,\ldots,$ are active.

By sequence of \textit{up-crossing} time instants $t_1^\prime$,
$t_2^\prime$,\ldots, $t_j^\prime$,\ldots, we mean
$$
\begin{aligned}
t_1^\prime=t_1, \quad &t_2^\prime=\inf\{t>t_1^\prime:
\|\mathbf{S}_t\|=\|\mathbf{S}_{t-1}\|+1\}, \ldots,\\
&t_{j}^\prime=\inf\{t>t_{j-1}^\prime:\|\mathbf{S}_t\|=\|\mathbf{S}_{t-1}\|+1\},\ldots
\end{aligned}
$$

In the following considerations, without loss of generality we
assume that the time intervals between active time instants are
exponentially distributed.

Specifically, with any state $\mathbf{n}\in\mathbb{Z}^d$ we
associate $2d$ independent Poisson processes. As a random walk is in
state $\mathbf{n}$, these Poisson processes define the direction of
the following movement of that random walk. For instance, in the
case of Model 1, for any state $\mathbf{n}\in\mathbb{Z}^d$ the
Poisson processes have the rates
\begin{equation}\label{eq-1.9}
{\alpha_1}, {\alpha_1}, {\alpha_2},
{\alpha_2},\ldots,\underbrace{\alpha_i}_{+\mathbf{1}_i},
\underbrace{\alpha_i}_{-\mathbf{1}_i},\ldots, {\alpha_d},
{\alpha_d},
\end{equation}
where the notation $\underbrace{\alpha_i}_{+\mathbf{1}_i}$ indicates
that the Poisson process with rate $\alpha_i$ is associated with the
direction $\mathbf{1}_i$, and, respectively, the notation
$\underbrace{\alpha_i}_{-\mathbf{1}_i}$ indicates that the Poisson
process with rate $\alpha_i$ is associated with the direction
$(-\mathbf{1}_i)$. So, in \eqref{eq-1.9} the rates with odd order
number in the row are associated with ``positive" unit direction,
while the rates with even order number with ``negative". Then, the
time between consecutive jumps is exponentially distributed with
mean $2\sum_{i=1}^d\alpha_i=1$, and the direction of the jump is
associated with the location of the minimum of the $2d$
exponentially distributed ``inter-jump" times associated with the
aforementioned Poisson processes.

In the case of Models 2 or 3, the rates of Poisson processes are
state dependent. If in the primary scale
$\mathsf{E}(t_{j+1}-t_j)=\nu$, then the length $t_{j+1}-t_j$ in the
new scale is to be exponentially distributed with mean $\nu$.

Let $P_{t_j}^{(a)}(n)$, $n\geq0$, $a\in\mathcal{A}$, denote the
transition probability
$$
P_{t_j}^{(a,d)}(n)=\mathsf{P}\big\{\|\mathbf{S}_{t_{j+1}}({a},d)\|=n+1
~\big|~\|\mathbf{S}_{t_j}({a},d)\|=n\big\}.
$$
Then,
\begin{equation}\label{eq-1.6}
p^{(a,d)}(n)=\lim_{j\to\infty}P_{t_j}^{(a,d)}(n)
\end{equation}
for all $n\geq0$ and $a\in\mathcal{A}$ is the notation for this
limiting probability.

\begin{defn}
The family of random walks $\{\mathbf{S}_t,\mathcal{A},d\}$ is
called \textit{conservative} (or
$(\mathcal{A},d)$-\textit{conservative}), if for any
$a_1\in\mathcal{A}$ and $a_2\in\mathcal{A}$ and all $n\geq0$
\begin{equation*}\label{eq-1.3}
p^{(a_1,d)}(n)=p^{(a_2,d)}(n)\equiv p^{(d)}(n),
\end{equation*}
and \textit{semiconservative}
(($\mathcal{A},d)$-\textit{semiconservative}) if there exists
$a^*\in\mathcal{A}$ such that for all $a\in\mathcal{A}$ and $n\geq0$
either
\begin{equation*}\label{eq-1.4}
p^{(a^*,d)}(n)\leq p^{(a,d)}(n),
\end{equation*}
or
\begin{equation*}\label{eq-1.5}
p^{(a^*,d)}(n)\geq p^{(a,d)}(n).
\end{equation*}
\end{defn}

The rest of the paper is organized as follows. In Section \ref{S2},
we study the family of symmetric random walks of Model 1 and prove
that it is $(\mathcal{A},d)$-conservative, and hence the limit in
\eqref{eq-1.6} describes a unique sequence $p^{(a,d)}(n)\equiv
p^{(d)}(n)$ for all $a\in\mathcal{A}$ and any $d$. The analysis in
Section \ref{S3} is related to the families of random walks of
Models 2 and 3. It is shown that the families of random walks of
these models are semiconservative. In Section \ref{S4}, we study
recurrent and transient random walks. We prove that all previously
considered Models 1, 2 and 3 are recurrent for $d\leq2$ and
transient for $d\geq3$. We extend the results for a possibly most
general state-dependent model (Model 5). In the same section, we
give non-trivial examples of transient two-dimensional random walks
and recurrent $d$-dimensional random walks for $d\geq3$ (which
reduce to the systems of specifically defined independent
null-recurrent birth-and-death processes). In Section \ref{S5}, we
discuss the results and conclude the paper.

\section{Symmetric random walks}\label{S2}

In this section we study symmetric random walks for Model 1. Let
$\{\mathbf{S}_t,\mathcal{A},d\}$ denote a family of symmetric random
walks in $\mathbb{Z}^d$, and let $\mathbf{S}_t(a,d)$ be a random
walk, $a=(\alpha_1,\alpha_2,\ldots,\alpha_d)$.

Let $\varphi_t(a,d)=\|\mathbf{S}_t(a,d)\|$. The functional
$\varphi_t(a,d)$ satisfies the following recurrence relations
\begin{eqnarray}
\varphi_0(a,d)&=&0,\label{eq-2.1}\\
\varphi_t(a,d)&=&|\varphi_{t-1}(a,d)+X_t(\varphi_{t-1}(a),a,d)|,
\quad t\geq1,\label{eq-2.2}
\end{eqnarray}
where
\begin{equation}\label{eq-2.3}
X_t(n,a,d)=\begin{cases}+1, &\text{with probability} \quad
P_t^{(a,d)}(n),\\
-1, &\text{with probability} \quad Q_t^{(a,d)}(n)=1-P_t^{(a,d)}(n).
\end{cases}
\end{equation}

Let $BD(\gamma,d)$ denote the family of birth-and-death processes,
$\gamma>0$, $d\geq1$ ($d$ is the integer-valued parameter associated
with the dimension of random walks), with the birth and death
parameters $\lambda_n(\gamma,d)$ and $\mu_n(\gamma,d)$,
respectively, where
$$
\lambda_n(\gamma,d)=\sum_{i=1}^d
i\gamma^{i}\binom{d}{i}\binom{n-1}{i-1}+\gamma\sum_{i=1}^{d-1}(d-i)
\gamma^{i}\binom{d}{i}\binom{n-1}{i-1},
$$
and
$$
\mu_n(\gamma,d)=\sum_{i=1}^d
i\gamma^{i}\binom{d}{i}\binom{n-1}{i-1},
$$
$n\geq1$. (The parameter $\lambda_0(\gamma,d)>0$ is not used and
hence can be taken arbitrarily.)

Notice that with $d=1$ the family $BD(\gamma,1)$ is trivial, since
$\lambda_n(\gamma,1)=\mu_n(\gamma,1)$. Therefore, in the sequel the
only case $d\geq2$ is considered.

\smallskip

We prove the following result.
\begin{thm}\label{T1}
The family of random walks $\{\mathbf{S}_t,\mathcal{A},d\}$ is
$(\mathcal{A},d)$-conservative. It is characterized by the family of
birth-and-death processes $BD(2,d)$. Specifically,
$$p^{(a,d)}(n)\equiv p^{(d)}(n)=\frac{\lambda_n(2,d)}{\lambda_n(2,d)+\mu_n(2,d)}$$ for
all $a\in\mathcal{A}$ and $d\geq2$.
\end{thm}

\begin{proof} For convenience, the
arguments $a$ and $d$ in the following notation for
$\mathbf{S}_t(a,d)$ is omitted, since in this proof we deal with a
unique random walk.
 Along with the original random walk $\mathbf{S}_t$ in
$\mathbb{Z}^d$, one can consider the reflected random walk
$\breve{\mathbf{S}}_t=\Big(\breve{S}_t^{(1)}$,$\breve{S}_t^{(2)}$,\ldots,
$\breve{S}_t^{(d)}\Big)$ in $\mathbb{Z}^d_+$ defined as
\begin{eqnarray}
\breve{\mathbf{S}}_0&=&\mathbf{0},\label{eq-3.10}\\
\breve{\mathbf{S}}_t&=&\breve{\mathbf{S}}_{t-1}+\mathbf{r}_t,\ \quad
t\geq1,\label{eq-3.11}
\end{eqnarray}
where
\begin{equation*}
{\mathbf{r}}_t=\begin{cases} \mathbf{e}_t, &\text{if}\
 \breve{S}_{t-1}^{(i)}+e_t^{(i)}\geq0 \ \text{for all} \ i=1,2,\ldots,d,\\
-\mathbf{e}_t, &\text{if}\ \breve{S}_{t-1}^{(i)}+e^{(i)}_t=-1 \
\text{for a certain} \ i=1,2,\ldots,d,
\end{cases}
\end{equation*}
and the vector $\mathbf{e}_t=\big(e_t^{(1)}$, $e_t^{(2)}$,\ldots,
$e_t^{(d)}\big)$ is the vector that was defined earlier for the
random walk $\mathbf{S}_t$ in \eqref{eq-1.1} and \eqref{eq-1.2}.

Observing \eqref{eq-3.10}, \eqref{eq-3.11} and their comparison with
\eqref{eq-1.1}, \eqref{eq-1.2} enables us to conclude that
$\|\breve{\mathbf{S}}_t\|$ and $\|\mathbf{S}_t\|$ coincide in
distribution. So, one can restrict ourselves by modeling the random
walk $\breve{\mathbf{S}}_t$. The last one is described by the $d$
independent queueing processes as follows. Assume that arrivals in
the $i$th queueing system are Poisson with rate $\alpha_i$, and
service times are exponentially distributed with the same rate
$\alpha_i$. If a system becomes free, it is switched for a special
service with the same rate $\alpha_i$. This service is
\textit{negative}, and it results in a new customer in the queue. If
during a \textit{negative service} a new arrival occurs, the
negative service remains unfinished and not resumed.

The negative service models the reflection at zero and in fact
implies the state-dependent arrival rate, which becomes equal to
$2\alpha_i$ at the moment when the system is empty. It is associated
with the situation, when an original one-dimensional random walk
reaches zero at some time moment $s$, and at the next time moment
$s+1$ it must take one of values $\pm1$ that corresponds to value +1
for an one-dimensional random walk \textit{reflected at zero}.

To establish necessary properties of the $d$ independent queueing
systems, assume first that the number of waiting places in each of
the queueing systems is $N$, where $N$ is taken large enough, such
that for a given $d$-dimensional vectors $\mathbf{n}\in
\mathbb{Z}^d_+$, we have $\|\mathbf{n}\|<N$. The assumption on
limited number of waiting places means that an arriving customer,
who meets $N$ customers in the system, is lost. Let
$P_N(\mathbf{n})$ denote the stationary probability to be in state
$\mathbf{n}$ immediately before an arrival of a customer in one of
the $d$ queueing systems. Following the PASTA property \cite{Wolff},
the stationary probabilities can be derived on the basis of the
backward Chapman-Kolmogorov equations for a continuous Markov
process (the random walk $\breve{\mathbf{S}}_t$ is reckoned to be
extended to the continuous time process). Calculate first the
stationary probability for a single queueing system. Denote by
$P_N^{(i)}(n)$ the stationary probability to be in state $n$ in the
$i$th queueing system. From the balance equations, we obtain
$P_N^{(i)}(n+1)=P_N^{(i)}(n)$ for $n\geq1$, and $P_N(1)=2P_N(0)$.
So, we arrive at
\begin{equation}\label{eq-3.12}
P_N^{(i)}(n)=\begin{cases}\frac{2}{2N+1}, &\text{for} \ 1\leq n\leq N,\\
\frac{1}{2N+1}, &\text{for} \ n=0.
\end{cases}
\end{equation}
Notice that the stationary probabilities do not depend on the rate
$\alpha_i$. Since the queueing systems are independent, we arrive at
the product form solution
\begin{equation}\label{eq-3.13}
P_N(\mathbf{n})=\prod_{i=1}^dP_N^{(i)}\left(n^{(i)}\right).
\end{equation}
Substituting \eqref{eq-3.12} into \eqref{eq-3.13} we may obtain the
exact form solution. For any vector $\mathbf{n}\in\mathbb{Z}^d_+$,
let $d_0(\mathbf{n})$ denote the number of zero components in the
presentation of the vector $\mathbf{n}$. For example, vector
(1,0,2,0) contains two zero components. Then,
\begin{equation}\label{eq-3.14}
P_N(\mathbf{n})=2^{d-d_0(\mathbf{n})}\frac{1}{(2N+1)^d}
\end{equation}
is the required formula for the stationary probability, and for two
arbitrary states $\mathbf{n}_1$ and $\mathbf{n}_2$, the ratio
$$
\frac{P_N(\mathbf{n}_1)}{P_N(\mathbf{n}_2)}=2^{d_0(\mathbf{n}_2)
-d_0(\mathbf{n}_1)}
$$
does not depend on $N$.  Hence,
\begin{equation}\label{eq-3.5}
\lim_{N\to\infty}\frac{P_N(\mathbf{n}_1)}{P_N(\mathbf{n}_2)}=2^{d_0(\mathbf{n}_2)
-d_0(\mathbf{n}_1)}.
\end{equation}
Next, denote
\begin{equation}\label{eq-3.18}
\Pi_{\infty}(\mathbf{n}_1,\mathbf{n}_2)=\lim_{j\to\infty}
\frac{\mathsf{P}\left\{\breve{\mathbf{S}}_{t_j^\prime-1}=\mathbf{n}_1\right\}}
{\mathsf{P}\left\{\breve{\mathbf{S}}_{t_j^\prime-1}=\mathbf{n}_2\right\}}.
\end{equation}
Here, in \eqref{eq-3.18}, the subindex $t_j^\prime-1$ denotes the
time instant preceding the $j$th up-crossing time instant
$t_j^\prime$. As well, by the product rule and the PASTA property
$$
\Pi_{\infty}(\mathbf{n}_1,\mathbf{n}_2)=\prod_{i=1}^d\Pi_{\infty}^{(i)}
\big(n_1^{(i)},n_2^{(i)}\big)=\prod_{i=1}^d\lim_{j\to\infty}
\frac{\mathsf{P}\left\{\breve{{S}}_{t_j^\prime-1}^{(i)}={n}_1^{(i)}\right\}}
{\mathsf{P}\left\{\breve{S}_{t_j^\prime-1}^{(i)}={n}_2^{(i)}\right\}}.
$$
Note, that
$$\sum_{n=0}^\infty\lim_{j\to\infty}
\mathsf{P}\left\{\breve{{S}}_{t_j^\prime-1}^{(i)}=n\right\}
=\lim_{j\to\infty}\sum_{n=0}^\infty\mathsf{P}\left\{
\breve{{S}}_{t_j^\prime-1}^{(i)}=n\right\} =1, \quad i=1,2,\ldots,
d,
$$
and on the basis of the Chapman-Kolmogorov equations for large $j$
we obtain
$$
\mathsf{P}\left\{\breve{{S}}_{t_j^\prime-1}^{(i)}=1\right\}=2
\mathsf{P}\left\{\breve{{S}}_{t_j^\prime-1}^{(i)}=0\right\}[1+o(1)],
$$
and
$$
\mathsf{P}\left\{\breve{{S}}_{t_j^\prime-1}^{(i)}=n+1\right\}=
\mathsf{P}\left\{\breve{{S}}_{t_j^\prime-1}^{(i)}=n\right\}[1+o(1)]
$$
for $n\geq1$. Then, for $n_1^{(i)}>n_2^{(i)}$, we obtain
\begin{equation}\label{eq-3.19}
\lim_{j\to\infty}\frac{\mathsf{P}\left\{
\breve{{S}}_{t_j^\prime-1}^{(i)}=n_1^{(i)}\right\}}
{\mathsf{P}\left\{
\breve{{S}}_{t_j^\prime-1}^{(i)}=n_2^{(i)}\right\}}=
\begin{cases}1, &\text{if} \ n_2^{(i)}>0,\\
2, &\text{if} \ n_2^{(i)}=0.\\
\end{cases}
\end{equation}

Now, according to \eqref{eq-3.5}, \eqref{eq-3.18} and
\eqref{eq-3.19}, we arrive at
\begin{equation}\label{eq-3.25}
\Pi_{\infty}(\mathbf{n}_1,\mathbf{n}_2)=\lim_{N\to\infty}\frac{P_N(\mathbf{n}_1)}{P_N(\mathbf{n}_2)}=2^{d_0(\mathbf{n}_2)
-d_0(\mathbf{n}_1)}.
\end{equation}

Now, let $\mathcal{N}^+(n)$ denote the set of all vectors in
$\mathbb{Z}^d_+$ having norm $n$. The total number of vectors having
norm $n$ in $\mathbb{Z}^d_+$ is
\begin{equation}\label{eq-3.24}
\sum_{i=1}^d\binom{d}{i}\binom{n-1}{i-1}.
\end{equation}
Here and later on, the formula sums over $i$ being the number of
nonzero components of the vector.

Hence, denoting the stationary state probability to belong to the
set $\mathcal{N}^+(n)$ by $P_N[\mathcal{N}^+(n)]$, from
\eqref{eq-3.14} we obtain
\begin{equation}\label{eq-3.16}
P_N[\mathcal{N}^+(n)]=\sum_{\mathbf{n}\in\mathcal{N}^+(n)}P_N(\mathbf{n})
=\frac{1}{(2N+1)^d}\sum_{i=1}^d2^i\binom{d}{i}\binom{n-1}{i-1}.
\end{equation}
Here, in the right-hand side of \eqref{eq-3.16}, the term
$\sum_{i=1}^d2^i\binom{d}{i}\binom{n-1}{i-1}$   characterizes the
total number of elements in $\mathbb{Z}^d$ having norm $n$.

Let $p_n(d)$ denote the transition probability from the set of
states $\mathcal{N}^+(n)$ (level $n$) to the set of states
$\mathcal{N}^+(n+1)$ (level $n+1$), and let $q_n(d)=1-p_n(d)$ denote
the transition probability from the level $n$ to the level $n-1$.

Our task now is to derive the formula for $p_n(d)$, and for this
derivation we use combinatorial arguments.

The total number of vectors in the set $\mathcal{N}^+(n)$ is given
by \eqref{eq-3.24}. Each vector contains $d$ components. Hence, the
total number of components in the set of vectors in
$\mathcal{N}^+(n)$ is
\begin{equation}\label{eq-3.26}
d\sum_{i=1}^d\binom{d}{i}\binom{n-1}{i-1}.
\end{equation}
Among them, the total number of zero components is
\begin{equation*}\label{eq-3.1}
\sum_{i=1}^{d-1}(d-i)\binom{d}{i}\binom{n-1}{i-1},
\end{equation*}
and the total number of nonzero components in all the aforementioned
vectors is
\begin{equation*}\label{eq-3.2}
\sum_{i=1}^d i\binom{d}{i}\binom{n-1}{i-1}.
\end{equation*}

Based on \eqref{eq-3.14} or \eqref{eq-3.16}, it will be proved below
the following relationships
\begin{equation}\label{eq-6.2}
\begin{aligned}
p_n(d)&=\frac{2\sum_{i=1}^{d-1}(d-i)2^{i}\binom{d}{i}\binom{n-1}{i-1}+
\sum_{i=1}^{d}i2^{i}\binom{d}{i}\binom{n-1}{i-1}}
{2d\sum_{i=1}^{d}2^{i}\binom{d}{i}\binom{n-1}{i-1}}\\
&=\frac{C_0(n,d)+C(n,d)}{C_0(n,d)+2C(n,d)},
\end{aligned}
\end{equation}
and
$$
q_n(d)=\frac{C(n,d)}{C_0(n,d)+2C(n,d)},
$$
where
\begin{equation}\label{eq-2.4}
C(n,d)=\sum_{i=1}^{d}i2^{i}\binom{d}{i}\binom{n-1}{i-1},
\end{equation}
and
\begin{equation}\label{eq-2.5}
C_0(n,d)=2\sum_{i=1}^{d-1}(d-i)2^{i}\binom{d}{i}\binom{n-1}{i-1}.
\end{equation}
The plan of the proof is as follows. We first consider the case of
$\alpha_i=1/(2d)$ for all $i=1,2,\ldots,d$ (P\'olya random walk),
and then develop the proof for the general situation.

\smallskip
\textit{Case of P\'olya random walk.} First, we build the sample
space. The components of all vectors in $\mathcal{N}^+(n)$, the
total number of which is given by \eqref{eq-3.26} are not equally
likely. According to \eqref{eq-3.14}, a nonzero component appears
with two times higher probability than a zero component. To make the
components equally likely, we are to extend the number of nonzero
components by factor 2. Then the total number of equally likely
components is to be equal to
\begin{equation}\label{eq-3.27}
 d\sum_{i=1}^d2^i\binom{d}{i}\binom{n-1}{i-1}.
\end{equation}
Following \eqref{eq-3.27}, the sample space for level $n$ contains
$$
2d\sum_{i=1}^d2^i\binom{d}{i}\binom{n-1}{i-1}=C_0(n,d)+2C(n,d)
$$
states characterizing the number of possible transitions of the
vectors in $\mathbb{Z}^d$ having norm $n$, where $C(n,d)$ and
$C_0(n,d)$ are given by \eqref{eq-2.4} and \eqref{eq-2.5}.
Specifically, $C_0(n,d)$ is the number of possible transitions
associated with reflections at zero. The rest of all possible
transitions of the sample space is $2C(n,d)$. Half of them
characterize transitions from level $n$ to $n+1$ and half from level
$n$ to $n-1$. Hence, the total number of transitions from level $n$
to level $n+1$ is $C_0(n,d)+C(n,d)$, and we arrive at
\eqref{eq-6.2}. So, in the case $\alpha_i=1/(2d)$, $i=1,2,\ldots,d$
relation \eqref{eq-6.2} is explained.

\smallskip
\textit{General case.} We prove now, that \eqref{eq-6.2} remains
true in the general case of $a=(\alpha_1,\alpha_2,\ldots,\alpha_d)$,
in which the probability of a transition of the $j$th coordinate of
a vector is
$$r_j=\frac{\alpha_j}{\alpha_1+\alpha_2+\ldots+\alpha_d}.$$

Since, the set of states has the symmetric structure, then the
contribution of the different terms $r_j$, $j=1,2,\ldots,d$ changes
the terms $C_0(n,d)$ and $C(n,d)$ proportionally. Indeed, consider
the $i$th term of the sum in \eqref{eq-2.4}. It is associated with
$i2^i\binom{d}{i}\binom{n-1}{i-1}$ elements. Assuming that the
choice of the $j$th coordinate of the vector $\mathbf{n}$  has the
weight $r_j$, it can be seen that the fraction of the terms
(weights) $r_j$ for different $j$ is the same among the total
$i2^i\binom{d}{i}\binom{n-1}{i-1}$ elements. That is, the terms
$r_j$ all are uniformly concentrated among
$i2^i\binom{d}{i}\binom{n-1}{i-1}$ elements, and this is true for
all $i$. In this case, instead of the term $C(n,d)$ we have the
modified term (weighted sum) denoted by $\tilde{C}(n,d)$. Then, we
write the relation
\begin{equation}\label{eq-3.3}
\tilde{C}(n,d)=\tilde{c}(n,d)C(n,d).
\end{equation}
However, the similar arguments are valid for the $i$th terms of the
sum in \eqref{eq-2.5}. Hence, denoting the corresponding weighted
sum by $\tilde{C}_0(n,d)$, we have
\begin{equation}\label{eq-3.4}
\tilde{C}_0(n,d)=\tilde{c}(n,d)C_0(n,d)
\end{equation}
with the same proportion coefficient $\tilde{c}(n,d)$. That is,
instead of $C_0(n,d)$ and $C(n,d)$ we are to have $\tilde{C}(n,d)$
and $\tilde{C}_0(n,d)$ defined by \eqref{eq-3.3} and \eqref{eq-3.4},
respectively. Hence, the constant $\tilde{c}(n,d)$ being presented
both in numerator and denominator of \eqref{eq-6.2} finally reduces,
and the probabilities $r_j$, $j=1,2,\ldots,d$, have no influence on
the parameters $p_n(d)$ and $q_n(d)$.
\smallskip

As we can see, these transition probabilities depend neither on $N$
nor on parameters $\alpha_i$, $i=1,2,\ldots,d$, but on $d$ only.
Hence, as $N$ increases to infinity, the parameters $p_n(d)$ and
$q_n(d)$ remain unchanged. This means that for any
$a\in\mathcal{A}$, the limiting relation in \eqref{eq-1.6} is
satisfied, and with \eqref{eq-6.2} the limiting birth-and-death
process is $BD(2,d)$. Thus, the family of random walks
$\{\mathbf{S}_t,\mathcal{A},d\}$ is $(\mathcal{A},d)$-conservative.
\end{proof}

\section{State-dependent random walks of Models 2 and 3}\label{S3}

In this section we prove the following results.

\begin{thm}\label{T2}
The family of random walks $\{\mathbf{S}_t,\mathcal{A},d\}$ defined
in Models 2 and 3 are semiconservative.
\end{thm}

\begin{proof}
The idea of the proof for each of Models 2 and 3  is similar to that
given in the proof of Theorem \ref{T1} for Model 1. Consider first
the family of random walks of Model 2. We are to consider the
reflected version of the random walk (denoted by
$\breve{\mathbf{S}}_t$) and model it as $d$ state-dependent queueing
systems, which according to their construction are independent of
each other. Let $\mathbf{n}=\big(n^{(1)}$, $n^{(2)},\ldots,$
$n^{(d)}\big)$, $n^{(i)}\geq0$, be a vector. Denote the
state-dependent arrival and service rates for the $i$th queueing
system, since they are equal, both by
$\tilde{\beta}_i\big(n^{(i)}\big)$. That is, if $n^{(i)}\geq1$, then
$\tilde{\beta}_i\big(n^{(i)}\big)=\alpha_i-\delta_i/2$,
$i=1,2,\ldots,d$. Otherwise, if $n^{(i)}=0$, then the arrival rate
$\tilde{\beta}_i(0)=2\alpha_i$. The role of coefficient 2 (negative
service doubles arrival rate) is explained in the proof of Theorem
\ref{T1}. As in the proof of Theorem \ref{T1} assume first that the
number of waiting places is $N$, where $N$ is taken larger than the
norm of vector $\mathbf{n}$, that is, $N>\|\mathbf{n}\|$. Let
$P_N(\mathbf{n})$ denote the stationary probability. (The details of
the definition are the same as in the proof of Theorem \ref{T1}.) To
derive the expression for $P_N(\mathbf{n})$, consider first an $i$th
queueing system, and denote by $P_N^{(i)}(n)$ the stationary
probability to be in the state $n$. We have the following
relationships. For $n\geq1$, $P_N^{(i)}(n+1)=P_N^{(i)}(n)$, and
$P_N^{(i)}(1)=[4\alpha_i/(2\alpha_i-\delta_i)]P_N^{(i)}(0)$. Then,
\begin{equation}\label{eq-2.6}
P_N^{(i)}(n)=\begin{cases}\frac{4\alpha_i/(2\alpha_i-\delta_i)}
{N[4\alpha_i/(2\alpha_i-\delta_i)]+1}, &\text{for} \quad 1\leq
n\leq N,\\
\frac{1}{N[4\alpha_i/(2\alpha_i-\delta_i)]+1}, &\text{for} \quad
n=0,
\end{cases}
\end{equation}
and since the queueing systems are independent,
\begin{equation}\label{eq-2.7}
P_N(\mathbf{n})=\prod_{i=1}^d P_N^{(i)}\left(n^{(i)}\right)=
\prod_{i=1}^d\frac{[4\alpha_i/(2\alpha_i-\delta_i)]^{\min\{n^{(i)},1\}}}
{4N\alpha_i/(2\alpha_i-\delta_i)+1}.
\end{equation}

Similarly to \eqref{eq-3.16}, the level $n$ probability in given by
$$
P_N\left[\mathcal{N}^+(n)\right]=\sum_{\mathbf{n}\in\mathcal{N}^+(n)}P_N(\mathbf{n}).
$$
where $P_N(\mathbf{n})$ are given by \eqref{eq-2.7}. Denote the
transition probabilities from the set of states $\mathcal{N}^+(n)$
to the set of states $\mathcal{N}^+(n+1)$ by
$p_n(d,k_1,k_2,\ldots,k_d)$, where
$k_i=2\alpha_i/(2\alpha_i-\delta_i)$. The explicit representation
for $p_n(d,k_1,k_2,\ldots,k_d)$ is cumbersome, but in the particular
case where $\delta_i=0$, $i=1,2,\ldots,d$, it coincides with
\eqref{eq-6.2}. As $\delta_i$ increases, the transition probability
$p_n(d,k_1,k_2,\ldots,k_d)$ must either increase of decrease in
dependence of the value $d$. The straightforward analytical proof of
this based on the induction in $d$ takes much place. Instead, we
prove a more particular statement. Let
$$
\underline{k}=\min_{1\leq i\leq
d}\frac{2\alpha_i}{2\alpha_i-\delta_i}
$$
and
$$
\overline{k}=\max_{1\leq i\leq
d}\frac{2\alpha_i}{2\alpha_i-\delta_i}.
$$

We prove the justice of the following chain of the inequalities
\begin{equation}\label{eq-3.20}
p_n(d,\underline{k})\leq p_n(d,\underbrace{k,k,\ldots,k}_{d \
\text{times}})\leq p_n(d,\overline{k}),
\end{equation}
assuming that $k_1=k_2=\ldots=k_d=k$, $\underline{k}\leq k\leq
\overline{k}$, where, based on the arguments in the proof of Theorem
\ref{T1} (see the explanation for relation \eqref{eq-6.2}), the
explicit representations for $p_n(d,\underbrace{k,k,\ldots,k}_{d \
\text{times}})$ as well as the lower and upper bounds
$p_n(d,\overline{k})$ and $p_n(d,\underline{k})$ are
\begin{equation}\label{eq-3.23}
p_n(d,\underbrace{k,k,\ldots,k}_{d \ \text{times}})=
\frac{2k\sum_{i=1}^{d-1}(d-i)(2k)^{i}\binom{d}{i}\binom{n-1}{i-1}+
\sum_{i=1}^{d}i(2k)^{i}\binom{d}{i}\binom{n-1}{i-1}}
{2k\sum_{i=1}^{d-1}(d-i)(2k)^{i}\binom{d}{i}\binom{n-1}{i-1}+
2\sum_{i=1}^{d}i(2k)^{i}\binom{d}{i}\binom{n-1}{i-1}},
\end{equation}
\begin{equation}\label{eq-3.21}
p_n(d,\overline{k})=
\frac{2\overline{k}\sum_{i=1}^{d-1}(d-i)(2\overline{k})^{i}\binom{d}{i}\binom{n-1}{i-1}+
\sum_{i=1}^{d}i(2\overline{k})^{i}\binom{d}{i}\binom{n-1}{i-1}}
{2\overline{k}\sum_{i=1}^{d-1}(d-i)(2\overline{k})^{i}\binom{d}{i}\binom{n-1}{i-1}+
2\sum_{i=1}^{d}i(2\overline{k})^{i}\binom{d}{i}\binom{n-1}{i-1}},
\end{equation}
and
\begin{equation}\label{eq-3.22}
p_n(d,\underline{k})=
\frac{2\underline{k}\sum_{i=1}^{d-1}(d-i)(2\underline{k})^{i}\binom{d}{i}\binom{n-1}{i-1}+
\sum_{i=1}^{d}i(2\underline{k})^{i}\binom{d}{i}\binom{n-1}{i-1}}
{2\underline{k}\sum_{i=1}^{d-1}(d-i)(2\underline{k})^{i}\binom{d}{i}\binom{n-1}{i-1}+
2\sum_{i=1}^{d}i(2\underline{k})^{i}\binom{d}{i}\binom{n-1}{i-1}}.
\end{equation}
Indeed, we are to prove that the derivative in $k$ for the
right-hand side of \eqref{eq-3.23} is strictly positive. Indeed,
following \eqref{eq-3.23} we have the representation
$$
p_n(d,\underbrace{k,k,\ldots,k}_{d \
\text{times}})=\frac{f(2k)+1}{f(2k)+2},
$$
where
$$
f(k)=\frac
{\sum_{i=1}^{d-1}(d-i)k^{i+1}\binom{d}{i}\binom{n-1}{i-1}}
{\sum_{i=1}^dik^i\binom{d}{i}\binom{n-1}{i-1}}.
$$
Hence, the task is to show that $f(k)$ is a decreasing function.
Taking the derivative of $f(k)$, we are to show the inequality
\begin{equation}\label{eq-3.17}
\begin{aligned}
&\sum_{i=1}^dik^i\binom{d}{i}\binom{n-1}{i-1}\sum_{j=2}^d
j(d-j+1)k^{j-1}\binom{d}{j-1}\binom{n-1}{j-2}\\
<&\sum_{i=1}^di^2k^{i-1}\binom{d}{i}\binom{n-1}{i-1}
\sum_{j=2}^d(d-j+1)k^j\binom{d}{j-1}\binom{n-1}{j-2}.
\end{aligned}
\end{equation}
The justice of \eqref{eq-3.17} follows from the fact that for
$j+k=2i$, we have $jk\leq i^2$. Thus the chains of inequalities
\eqref{eq-3.20} is correct. This constitutes that the family of
random walks of Model 2 is semiconservative, since the obtained
results remain unchanged in the limit as $N$ tends to infinity.

When $2\alpha_i/(2\alpha_i-\delta_i)=\overline{k}$ for all
$i=1,2,\ldots,d$, the family of random walks is associated with
$BD(2\overline{k},d)$, and when
$2\alpha_i/(2\alpha_i-\delta_i)=\underline{k}$ for all
$i=1,2,\ldots,d$, the family of random walks is associated with
$BD(2\underline{k},d)$.

For the family of random walks of Model 3 the proof is similar. For
$P_N(\mathbf{n})$ we obtain
\begin{equation*}
P_N(\mathbf{n})=
\prod_{i=1}^d\frac{[(2\alpha_i-\delta_i)/\alpha_i]^{\min\{n^{(i)},1\}}}
{N(2\alpha_i-\delta_i)/\alpha_i+1},
\end{equation*}
$$
P_N\left[\mathcal{N}^+(n)\right]=\sum_{\mathbf{n}\in\mathcal{N}^+(n)}P_N(\mathbf{n}),
$$
and similarly to \eqref{eq-3.20} the chain of inequalities
$$
p_n(d,\underline{k})\leq p_n(d,\underbrace{k,k,\ldots,k}_{d \
\text{times}})\leq p_n(d,\overline{k}), \quad \underline{k}\leq
k\leq\overline{k},
$$
with \eqref{eq-3.23}, \eqref{eq-3.21} and \eqref{eq-3.22}, where
$\underline{k}$ and $\overline{k}$ are now redefined as
$$
\underline{k}=\min_{1\leq i\leq
d}\frac{2\alpha_i-\delta_i}{2\alpha_i}
$$
and
$$
\overline{k}=\max_{1\leq i\leq
d}\frac{2\alpha_i-\delta_i}{2\alpha_i}.
$$
The rest arguments are similar to those given in the proof for Model
2.
\end{proof}

\section{Transient and recurrent random walks}\label{S4}

This section consists of three parts. In Section \ref{S4.1} the
properties of the birth-and-death processes $BD(\gamma,d)$ are
studied. In Section \ref{S4.2} the properties of random walks from
Models 1, 2 and 3 are discussed and their extensions (Models 4 and
5) are studied. In Section \ref{S4.3} an application of the study to
independent systems of null-recurrent birth-and-death processes is
discussed.

\subsection{Birth-and-death processes}\label{S4.1}
We start from general birth-and-death processes with birth rates
$\lambda_n$ and death rates $\mu_n$ satisfying the properties
$\lambda_n>\mu_n$ and $\lim_{n\to\infty}\lambda_n/\mu_n=1$.

\begin{lem}\label{l0}
Let the birth and death rates $\lambda_n$ and $\mu_n$ satisfy the
properties $\lambda_n>\mu_n$ and
$\lim_{n\to\infty}\lambda_n/\mu_n=1$. Then, the birth-and-death
process is transient if and only if
\begin{equation}\label{eq-4.1}
\lim_{n\to\infty}\left(\frac{\lambda_n}{\mu_n}\right)^n=\mathrm{e}^z
\end{equation}
is satisfied for $z>1$ and null-recurrent if and only if
\eqref{eq-4.1} is satisfied for $z\leq1$.
\end{lem}

\begin{proof}
According to the known classification of birth-and-death processes
\cite{KarlinMcGregor}, a birth-and-death process is null-recurrent
if and only if
$\sum_{n=1}^\infty\prod_{j=1}^n\lambda_{j-1}/\mu_j=\infty$ and
$\sum_{n=1}^\infty\prod_{j=1}^n\mu_j/\lambda_j=\infty$ and transient
if and only if
$\sum_{n=1}^\infty\prod_{j=1}^n\lambda_{j-1}/\mu_j=\infty$ and
$\sum_{n=1}^\infty\prod_{j=1}^n\mu_j/\lambda_j<\infty$. Apparently,
the first condition
$\sum_{n=1}^\infty\prod_{j=1}^n\lambda_{j-1}/\mu_j=\infty$ is always
satisfied. Indeed, since $\lim_{n\to\infty}\lambda_n/\mu_n=1$, one
can assume that both sequences $\{\lambda_n\}$ and $\{\mu_n\}$ are
properly normalized such that, starting from large $N$, both of them
are less than and greater than some constants $C_1$ and $C_2$,
respectively. Then, $\prod_{j=1}^n\lambda_{j-1}/\mu_j$ does not
vanish as $n\to\infty$, and hence the required series diverges.
Next, $\sum_{n=1}^\infty\prod_{j=1}^n\mu_j/\lambda_j<\infty$ is
satisfied if and only if there exist constants $C$ and $\epsilon>0$
such that for $n$ large enough and all $N>n$,
\begin{equation}\label{eq-4.3}
\prod_{j=1}^N\frac{\mu_j}{\lambda_j}\leq\frac{C}{N^{1+\epsilon}}.
\end{equation}
In turn, $\sum_{n=1}^\infty\prod_{j=1}^n\mu_j/\lambda_j=\infty$ is
satisfied if and only if for any $\epsilon>0$ there exists $n$ large
enough such that for all $N>n$
\begin{equation}\label{eq-4.4}
\prod_{j=1}^N\frac{\mu_j}{\lambda_j}>\frac{1}{N^{1+\epsilon}}.
\end{equation}

It is not difficult to show that \eqref{eq-4.4} implies the
asymptotic expansion
\begin{equation}\label{eq-4.2}
\frac{\mu_n}{\lambda_n}\asymp1-\frac{1}{zn}
\end{equation}
for some $z\leq1$. Indeed, if
$$
\prod_{j=1}^n\frac{\mu_j}{\lambda_j}\asymp\frac{C}{n}
$$
for some constant $C$ and $n$ increasing to infinity, then it is
readily seen that we arrive at
$$\frac{\mu_n}{\lambda_n}\asymp1-\frac{1}{n}.$$ Hence, under
condition \eqref{eq-4.4} we obtain \eqref{eq-4.2} with $z\leq1$, and
similarly, under condition \eqref{eq-4.3} we obtain \eqref{eq-4.2}
with $z>1$. The obtained expansions imply the corresponding limits
in \eqref{eq-4.1}. The lemma is proved.
\end{proof}

\begin{lem}\label{l1}
The family of birth-and-death processes $BD(\gamma,d)$ is null
recurrent for $d\leq2$ and transient for $d\geq3$.
\end{lem}

\begin{proof}
As $n\to\infty$, we obtain:
\begin{equation}\label{eq-2.11}
\begin{aligned}
\frac{\lambda_n(\gamma,d)}{\mu_n(\gamma,d)}&= \frac{
d\gamma^{d}\binom{n-1}{d-2}+(d-1)d\gamma^{d-1}\binom{n-1}{d-2}
+d\gamma^{d}\binom{n-1}{d-1}} {(d-1)d\gamma^{d-1}\binom{n-1}{d-2}
+d\gamma^{d}\binom{n-1}{d-1}}
\left[1+O\left(\frac{1}{n}\right)\right]\\
&=\frac{(d-1+\gamma)\binom{n-1}{d-2} +\gamma\binom{n-1}{d-1}}
{(d-1)\binom{n-1}{d-2}
+\gamma\binom{n-1}{d-1}}\left[1+O\left(\frac{1}{n}\right)\right]\\
&=\frac{\frac{1}{\gamma}(d-1)(d-1+\gamma) +(n-d+1)}
{\frac{1}{\gamma}(d-1)^2
+(n-d+1)}\left[1+O\left(\frac{1}{n}\right)\right].
\end{aligned}
\end{equation}
Hence,
\begin{equation}\label{eq-2.12}
\lim_{n\to\infty}\left(\frac{\lambda_n(\gamma,d)}{\mu_n(\gamma,d)}\right)^n
=\mathrm{e}^{d-1}.
\end{equation}
Thus, by virtue of Lemma \ref{l0} we arrive at the conclusion that
$BD(\gamma,d)$ is recurrent for $d\leq2$ and transient for $d\geq3$.
The lemma is proved.
\end{proof}

\subsection{Families of random walks}\label{S4.2}

\subsubsection{The results for Models 1, 2 and 3}

\begin{prop}\label{T3}
The family of symmetric random walks
$\{\mathbf{S}_t,\mathcal{A},d\}$ is recurrent for $d\leq 2$ and
transient for $d\geq3$.
\end{prop}

\begin{proof}
It follows from Theorem \ref{T1} that the family of symmetric random
walks is $(\mathcal{A},d)$-conservative and associated with
$BD(2,d)$. Hence, owing to Lemma \ref{l1}, the family of symmetric
random walks is recurrent for $d\leq2$ and transient for $d\geq3$.
\end{proof}

\begin{prop}\label{T4}
The family of random walks in Models 2 and 3 is recurrent for
$d\leq2$ and transient for $d\geq3$.
\end{prop}

\begin{proof}
Indeed, according to Theorem \ref{T2} the families of random walks
in Model 2 or 3 are semiconservative. The birth probabilities of the
associated birth-and-death process are bounded by the two-sided
inequalities of the birth probabilities of $BD(2\underline{k},d)$
and $BD(2\overline{k},d)$ processes. According to Lemma \ref{l1}
both of these birth-and-death processes are recurrent for $d\leq2$
and transient for $d\geq3$, and hence, the associated families of
random walks are recurrent for $d\leq2$ and transient for $d\geq3$.
\end{proof}

\subsubsection{An extended model}

In this section we consider a general random walk, which is an
extension of the random walks in Models 1, 2 and 3.

\smallskip

\textit{Model 4}. We consider the family of random walks defined by
\eqref{eq-1.1} and \eqref{eq-1.2}. Assume that $\mathbf{e}_t$
depends on the state $\mathbf{S}_{t-1}$ as follows. It takes values
$\mathbf{1}_i$ or $(-\mathbf{1}_i)$ with probability
$\tilde{\alpha}_i\big(S_{t-1}^{(i)}\big)\geq c>0$, where
$\tilde{\alpha}_i\big(S_{t-1}^{(i)}\big)\leq\alpha_i$ and
$2\sum_{i=1}^d\alpha_i=1$, and $\mathbf{e}_t$ takes value
$\mathbf{0}$ with the complementary probability,
$1-2\sum_{i=1}^d\tilde{\alpha}_i\big(S_{t-1}^{(i)}\big)$. The value
$c<\min\{\alpha_1, \alpha_2,\ldots, \alpha_d\}$ is an arbitrarily
small positive value.

\begin{prop}\label{T5}
The family of random walks in Model 4 is recurrent for $d\leq2$ and
transient for $d\geq3$.
\end{prop}

\begin{proof} The proof is based on coupling arguments. Consider
first the following two auxiliary models. In the first model (call
it Model A1), for all $n\geq1$
\begin{equation*}\label{eq-4.10}
\tilde{\alpha}_i(n)=\alpha_i,\quad \text{and} \quad
2\sum_{i=1}^d\alpha_i=1,
\end{equation*}
and
\begin{equation*}\label{eq-4.11}
\tilde{\alpha}_i(0)=c.
\end{equation*}
In the second model (call it Model A2)
\begin{equation*}\label{eq-4.12}
\tilde{\alpha}_i(0)=\alpha_i,\quad \text{and} \quad
2\sum_{i=1}^d\alpha_i=1,
\end{equation*}
and for all $n\geq1$
\begin{equation*}\label{eq-4.13}
\tilde{\alpha}_i(n)=c.
\end{equation*}
Apparently, Model A1 is a version of Model 3, and Model A2 is a
version of Model 2. According to Proposition \ref{T4}, both of them
are recurrent for $d\leq2$ and transient for $d\geq3$. Then the
coupling arguments enable us to conclude that the same is true for
Model 4, and the statements of Proposition \ref{T5} follow.
\end{proof}

\subsubsection{Further extension of Model 4}
We start from an extension of Model 1 and then, on the basis of it,
we provide further extension of Model 4.

\smallskip
\textit{Model B1}. For random walks defined by \eqref{eq-1.1} and
\eqref{eq-1.2}, let the vector $\mathbf{e}_t$ be one of the $2d$
randomly chosen vectors $\{\pm \mathbf{1}_i, i=1,2,\ldots,d\}$ as
follows. For $S^{(i)}_{t-1}\geq0$, the probability that the vector
$\mathbf{1}_i$ will be chosen is
$\tilde{\alpha}_i\big(S^{(i)}_{t-1}\big)>0$, and the probability
that the vector $(-\mathbf{1}_i)$ will be chosen is
$\tilde{\beta}_i\big(S^{(i)}_{t-1}\big)>0$, where
$$
\tilde{\alpha}_i\big(S^{(i)}_{t-1}\big)+\tilde{\beta}_i\big(S^{(i)}_{t-1}\big)
=2\alpha_i,
\quad 2\sum_{i=1}^d\alpha_i=1.
$$
The further specifications of the probabilities
$\tilde{\alpha}_i\big(S^{(i)}_{t-1}\big)$ and
$\tilde{\beta}_i\big(S^{(i)}_{t-1}\big)$ are as follows. Generally,
we assume
\begin{equation}\label{eq-4.14}
\tilde{\alpha}_i\big(S^{(i)}_{t-1}\big)=\tilde{\beta}_i\big(-S^{(i)}_{t-1}\big),
\end{equation}
and if $\big|S^{(i)}_{t-1}\big|>M$, then
\begin{equation}\label{eq-4.19}
\tilde{\alpha}_i\big(S^{(i)}_{t-1}\big)=\tilde{\beta}_i\big(S^{(i)}_{t-1}\big)=\alpha_i.
\end{equation}
Note also that according to \eqref{eq-4.14},
$\tilde{\alpha}_i(0)=\tilde{\beta}_i(0)=\alpha_i$.

\begin{lem}\label{l2}
The random walk in Model B1 is recurrent for $d\leq2$ and transient
for $d\geq3$.
\end{lem}

\begin{proof} To simplify the derivations and make the results of calculations
observable, we prove this lemma for the particular random walk,
assuming that if $1\leq\big|S^{(i)}_{t-1}\big|\leq M$, then
$$
\tilde{\alpha}_i\big(S^{(i)}_{t-1}\big)=\alpha_i^*, \quad
\tilde{\beta}_i\big(S^{(i)}_{t-1}\big)=2\alpha_i-\alpha_i^*.
$$
In addition, we assume that $\alpha_i^*/(2\alpha_i-\alpha_i^*)=\rho$
is the same constant for all $i$. The made assumption does not
logically change the proof. Then the random walk is modeled by the
following queueing system.

\smallskip

Consider the following series of $d$ state-dependent Markovian
queueing systems. The arrival rate of each of $d$ mutually
independent queueing systems depends on queue-length as follows. If
the $i$th system is empty, the arrival rate is $2\alpha_i$.
Otherwise, is there is at least one customer in the system, then the
arrival rate is $\alpha_i$. The service rate depends on the
queue-length as follows. If immediately before a moment of service
began there are more than $M$ customers in the system, then the
service rate is $\alpha_i$. Otherwise, it is $\beta_i$. The values
of parameters $\beta_i$ are scaled on the basis of the original
proportion between $\alpha_i^*$ and $2\alpha_i-\alpha_i^*$.

The proof of this lemma is similar to that of the proof of Theorem
\ref{T1}. We choose $N$ large enough that is greater than $M$, and
consider first an $i$th queueing system with $N$ waiting places.
Using the notation similar to that in the proof of Theorem \ref{T1},
for that queueing system from the Chapman-Kolmogorov equations we
obtain $P_N^{(i)}(1)=2(\alpha_i/\beta_i)P_N^{(i)}(0)=2\rho
P_N^{(i)}(0)$, $P_N^{(i)}(n+1)=\rho P_N(n)$ for $n=0,1,\ldots, M-1$,
and $P_N^{(i)}(n+1)=P_N^{(i)}(n)$ for $n>M-1$, and then,
$$
P_N(\mathbf{n})=\prod_{i=1}^dP_{N}^{(i)}\left(n^{(i)}\right).
$$
Next, assuming that $\|\mathbf{n}\|<N$, we have
\begin{equation}\label{eq-4.15}
P_N\left[\mathcal{N}^+(n)\right]=\sum_{\mathbf{n}\in\mathcal{N}^+(n)}P_N(\mathbf{n}).
\end{equation}
Similarly to \eqref{eq-3.16}, the explicit presentation for the
probability $P_N\big[\mathcal{N}^+(n)\big]$ can be derived, first in
the case $\alpha_1=\alpha_2=\ldots=\alpha_d\equiv\alpha$ and then in
the general case. For large $n$ and $N>n$ we will derive the
expansion for $p_n(d)$ based on the following results. First, we
take into account that for any two vectors $\mathbf{n}_1$ and
$\mathbf{n}_2$, the components of which all are greater than $M$, we
have
$$
\frac{P_N(\mathbf{n}_2)}{P_N(\mathbf{n}_1)}=1.
$$
If all components of vector $\mathbf{n}_2$ are greater than $M$,
while there are $d_0$ zero components of vector $\mathbf{n}_1$, then
\begin{equation*}\label{eq-4.18}
\frac{P_N(\mathbf{n}_2)}{P_N(\mathbf{n}_1)}\geq
2^{d_0}\rho^{(M+1)d_0},
\end{equation*}
where the equality is satisfied in the only case where all other
$d-d_0$ components are greater than $M$. In addition, as
$N\to\infty$, similarly to \eqref{eq-3.25}
\begin{equation}\label{eq-4.17}
\Pi_\infty(\mathbf{n}_1,
\mathbf{n}_2)=\lim_{N\to\infty}\frac{P_N(\mathbf{n}_2)}{P_N(\mathbf{n}_1)}.
\end{equation}

Then, for sufficiently large $n$, \eqref{eq-4.15} is evaluated by
\begin{equation}\label{eq-4.16}
P_N\left[\mathcal{N}^+(n)\right]
\asymp\sum_{\big\{\mathbf{n}\in\mathcal{N}^+(n): \quad
\min\{n^{(1)},n^{(2)},\ldots,n^{(d)}\}>M\big\}}P_N(\mathbf{n}),
\end{equation}
since the contribution of the terms $P_N(\mathbf{n})$ with
$$
\mathbf{n}\in\mathcal{N}^+(n)\setminus\big\{\mathbf{n}\in\mathcal{N}^+(n):
\quad \min_{1\leq i\leq d}n^{(i)}>M\big\}
$$
is negligible in \eqref{eq-4.15} as $n\to\infty$. Hence, based on
\eqref{eq-4.16} and the earlier result in \eqref{eq-6.2}, for
sufficiently large $n$ and $N>n$ we have the estimate
$$
p_n(d)\asymp\frac{C_0(2\rho^{M+1}, n,d)+C(2\rho^{M+1},n,d)}
{C_0(2\rho^{M+1},n,d)+2C(2\rho^{M+1},n,d)},
$$
where
$$
C(\gamma,n,d)=\sum_{i=1}^di\gamma^i\binom{d}{i}\binom{n-1}{i-1},
$$
and
$$
C_0(\gamma,n,d)=\gamma\sum_{i=1}^d(d-i)\gamma^i\binom{d}{i}\binom{n-1}{i-1}.
$$
Hence, as $n\to\infty$, the probability $p_n(d)$ is asymptotically
equivalent with the birth probability $p(n)$ in the birth-and-death
process $BD(2\rho^{M+1},d)$. Thus, according to Lemma \ref{l1} the
random walk is recurrent for $d\leq2$ and transient for $d\geq3$.
The proof is completed.
\end{proof}

\medskip
\textit{Model 5}. We consider the family of random walks defined by
\eqref{eq-1.1} and \eqref{eq-1.2}. Assume that $\mathbf{e}_t$
depends on the state $\mathbf{S}_{t-1}$ as follows. It takes value
$\mathbf{1}_i$ with probability
$\tilde{\alpha}_i\big(S_{t-1}^{(i)}\big)\geq c>0$, value
$(-\mathbf{1}_i)$ with probability
$\tilde{\beta}_i\big(S_{t-1}^{(i)}\big)\geq c>0$, where
$\tilde{\alpha}_i\big(S_{t-1}^{(i)}\big)+\tilde{\beta}_i\big(S_{t-1}^{(i)}\big)
\leq2\alpha_i$, $2\sum_{i=1}^d\alpha_i=1$, and $\mathbf{e}_t$ takes
value $\mathbf{0}$ with the complementary probability
$1-\sum_{i=1}^d\big[\tilde{\alpha}_i\big(S_{t-1}^{(i)}\big)
+\tilde{\beta}_i\big(S_{t-1}^{(i)}\big)\big]$. The value
$c<\min\{\alpha_1, \alpha_2,\ldots, \alpha_d\}$ is an arbitrarily
small positive value. The further specifications of the
probabilities $\tilde{\alpha}_i\big(S^{(i)}_{t-1}\big)$ and
$\tilde{\beta}_i\big(S^{(i)}_{t-1}\big)$ are as follows. Generally,
we assume \eqref{eq-4.14}, and if $\big|S^{(i)}_{t-1}\big|>M$, then
\eqref{eq-4.19}.

\begin{thm}\label{T6}
The random walk in Model 5 is recurrent for $d\leq2$ and transient
for $d\geq3$.
\end{thm}

\begin{proof}
The construction of the proof is as follows. As Model B1 is an
extension of Model 1, the similar extensions (called Models B2 and
B3) can be constructed for Models 2 and 3, respectively, and the
proof of Lemma \ref{l2} can be adapted to new Models B2 and B3 as
well. Then, the statement of Theorem \ref{T6} is proved by the way
that is used to prove Proposition \ref{T5} based on coupling
arguments.
\end{proof}

\begin{rem}
Condition \eqref{eq-4.14} that describes Model B1 is technical. It
is used for reduction of the original random work of Model B1 to the
reflected random walk, which in turn is described by the queueing
system constructed in the proof. We reckon that the statement of
Lemma \ref{l2} and Theorem \ref{T6} might be correct for the more
general models that do not include this condition.
\end{rem}

\begin{rem}\label{r0}
The condition
$\tilde{\alpha}_i\big(S_{t-1}^{(i)}\big)+\tilde{\beta}_i
\big(S_{t-1}^{(i)}\big)\leq2\alpha_i$ for Model 5 as well as the
similar condition
$\tilde{\alpha}_i\big(S_{t-1}^{(i)}\big)\leq\alpha_i$ for Model 4
are important. They guarantee that the components $S_t^{(i)}$,
$i=1,2,\ldots,d$ in the corresponding random walks $\mathbf{S}_t$
are independent.
\end{rem}

\subsection{Further examples of recurrent and transient random
walk}\label{S4.3}

The random walks that are described by Model 5 are characterized as
follows. Let $\breve{\mathbf{S}}_t=\big(\breve{S}_t^{(1)},
\breve{S}_t^{(2)}, \ldots, \breve{S}_t^{(d)}\big)$ be the reflected
random walks. Under the assumption that the random walks stay in
each of their states for an exponentially distributed time, the
components $\breve{S}_t^{(i)}$, $i=1,2,\ldots,d$ are thought as the
null-recurrent birth and death processes with the birth rates
$\lambda_n$ and death rates $\mu_n$ satisfying the property
$\lambda_n=\mu_n$ for $n\geq M+1$. In this section we discuss more
general situation of the system of $d$ independent null-recurrent
birth-and-death processes.

\smallskip
\begin{example}\label{ex1}
Let $S_t^{(1)}$ and $S_t^{(2)}$ be two null-recurrent independent
birth-and-death processes, and let $S_0^{(1)}=S_0^{(2)}=0$. For
simplicity of our analysis, assume that the birth-and-death
processes $S_t^{(1)}$ and $S_t^{(2)}$ are identically distributed.
That is both of them are specified by the same birth rates $L_n$ and
death rates $M_n$.
\end{example}

Let $\tau=\inf\{t>0: S_{t+h}^{(1)}+S_{t+h}^{(2)}=0\}$, where $h>0$
is an arbitrary constant.
\begin{thm}\label{T7}
Assume that
\begin{equation*}
\lim_{n\to\infty}\left(\frac{L_n}{M_n}\right)^n>1.
\end{equation*}
Then $\mathsf{P}\{\tau<\infty\}<1$.
\end{thm}

\begin{proof}
The proof is similar to that of Theorem \ref{T1} and based on
asymptotic analysis similar to that provided in this section to
prove Lemma \ref{l1}. First, taking $N$ large, we consider two
independent Markovian queueing systems with $N$ waiting places. For
simplicity, we assume that service times are identically distributed
with rate 1, and interarrival times are identically distributed with
rate $1+c/N$, $c>0$ is some positive constant. The following
arguments of the proof are similar to those given in the proof of
Theorem \ref{T1}, where we derive the asymptotic expression for
$p_n(2)$ as $N\to\infty$, and then in the proof of Lemma \ref{l1},
where we derive
$\lim_{n\to\infty}\big[\lambda_n(\gamma,2)/\mu_n(\gamma,2)\big]^n$.
Note, that the asymptotic behaviour in \eqref{eq-2.11} does not
depend on $\gamma$. Denote
$$
p_n=\lim_{t\to\infty}
\mathsf{P}\left\{S_{t+1}^{(1)}+S_{t+1}^{(2)}=n+1~|~S_{t}^{(1)}+S_{t}^{(2)}=n\right\}.
$$
Taking in account that for any $0<\kappa<1$ and $n\to\infty$
$$
\left(\frac{L_{\lfloor \kappa n\rfloor}}{M_{\lfloor \kappa
n\rfloor}}\cdot\frac{L_{n-\lfloor \kappa n\rfloor}}{M_{n-\lfloor
\kappa n\rfloor}}\right)^n\asymp\left(\frac{L_n}{M_n}\right)^n
$$
($\lfloor a\rfloor$ denotes the integer part of $a$), as
$n\to\infty$ we obtain
$$
\left(\frac{p_n}{1-p_n}\right)^n\asymp\left(\frac{\lambda_n(1,2)}{\mu_n(1,2)}
\cdot \frac{L_{n}}{M_{n}}\right)^n,
$$
and hence,
$$
\lim_{n\to\infty}\left(\frac{p_n}{1-p_n}\right)^n=\mathrm{e}^{c+1}.
$$
Since $c>0$, then according to Lemma \ref{l0} we obtain
$\mathsf{P}\{\tau<\infty\}<1$.
\end{proof}

\medskip

\begin{example}\label{ex2}
 Let $S_t^{(1)}$, $S_t^{(2)},\ldots$, $S_t^{(d)}$ \ $(d\geq3)$, be
independent, identically distributed, null-recurrent birth-and-death
processes. Let $$\tau=\inf\left\{t>0:
\sum_{i=1}^dS_{t+h}^{(i)}=0\right\},$$ where $h>0$ is an arbitrary
constant. Denote the birth and death rates by $L_n$ and $M_n$,
respectively.
\end{example}

\begin{thm}\label{T8}
Assume that
\begin{equation}
\lim_{n\to\infty}\left(\frac{L_n}{M_n}\right)^n\leq
\mathrm{e}^{2-d}.
\end{equation}
Then $\mathsf{P}\{\tau<\infty\}=1$.
\end{thm}

\begin{proof}
Using the similar arguments as in the proof of Theorem \ref{T7}, we
have as follows. Let
$$
p_n=\lim_{t\to\infty}
\mathsf{P}\left\{\sum_{i=1}^dS_{t+1}^{(i)}=n+1~|~\sum_{i=1}^dS_{t}^{(i)}=n\right\}.
$$
As $n\to\infty$ we obtain
$$
\left(\frac{p_n}{1-p_n}\right)^n\asymp\left[\frac{\lambda_n(1,2)}{\mu_n(1,2)}
\cdot\frac{L_{n}}{M_{n}}\right]^n.
$$
Hence,
$$
\lim_{n\to\infty}\left(\frac{p_n}{1-p_n}\right)^n\leq\mathrm{e}^{d-1}
\mathrm{e}^{2-d}=\mathrm{e}.
$$
Then, the statement of the theorem follows from Lemma \ref{l0}.
\end{proof}

\section{Discussion and concluding remarks}\label{S5}
 In the present paper, we gave a new classification of
multidimensional random walks. Based on that classification, we
established new results on the behaviour of random walks. The main
techniques used in the paper are reduction to birth-and-death
processes, asymptotic analysis and coupling arguments. The concepts
of conservative and semiconservative random walks are of independent
interest. The principally new results of the paper include the
analysis of Examples \ref{ex1} and \ref{ex2} resulted in the proof
of Theorems \ref{T7} and \ref{T8}. The statement of Proposition
\ref{T3} was previously covered by the results in Chung and Fuchs
\cite{ChungFuchs} (see also \cite{ChungOrnstein}) and Foster and
Good \cite{FosterGood}. A version of the proof of Chung and Fuchs
theorem is presented in Durrett \cite{Durett}. Specifically, Theorem
4.2.8 on page 166 and Theorem 4.2.13 on page 170 together claim that
any unbiased random walk in $\mathbb{R}^d$ having increments in the
domain of attraction of a Gaussian distribution is transient if and
only if $d\geq3$. The classes of random walks in \cite{ChungFuchs}
and \cite{FosterGood}, however, do not cover state-dependent random
walks considered in Models 2, 3, 4 and 5. MacPhee and Manshikov
\cite{MacPheeMenshikov} showed that a nonzero drift of a random walk
on a lower-dimensional subspace is sufficient in order to change the
recurrence classification. The method of Lyapunov functions that is
used by Lamperti \cite{Lamperti} provides intuition for the phase
transition. In its simplest version, the idea is to consider the
process $\varphi_t=\|\mathbf{S}_t\|=\Big[\sum_{i=1}^d\big(
S_t^{(i)}\big)^2\Big]^{1/2}$, the recurrence or transience of which
is determined by comparing
$$
\mathsf{E}\{\varphi_{t+1}-\varphi_{t}| \mathbf{S}_t=\mathbf{x}\}
$$
and
$$
\mathsf{E}\{(\varphi_{t+1}-\varphi_{t})^2|
\mathbf{S}_t=\mathbf{x}\}.
$$
It would be interesting to investigate the applicability of the
method in \cite{Lamperti} to the models under consideration here.

The results by Doyle and Snell \cite{DoyleSnell} also concern
state-dependent random walks similar to those described by Model B1
in the framework of electric networks theory. To be specific, we
refer the arXiv version of the book, where the relevant results are
in Section 2.4 ``Random walks on more general infinite networks",
page 101. The formulation and proof of the basic theorem is given on
page 102. The formulated theorem violates the conditions mentioned
in Remark \ref{r0}. Unfortunately, we could not follow the proof of
that theorem.

\section*{Acknowledgement} The author thanks the anonymous referee for
valuable comments leading to substantial improvement of the paper.

\end{document}


\title[]{Correction: Conservative and semiconservative random walks: Recurrence and transience}%
\author{Vyacheslav M. Abramov}%
\address{Monash University, School of Mathematical Sciences, Wellington Road, Clayton Campus, Clayton, Victoria-3800, Australia}%

\email{vabramov126@gmail.com}%

\maketitle

\noindent
\textbf{Correction to: J Theor Probab (2018) 31:1900–-1922}\\
\noindent
\textbf{ https://doi.org/10.1007/s10959-017-0747-3}

\medskip

The aim of this note is to correct the errors in the formulation and proof of Lemma 4.1 in \cite{A} and some claims that are based on that lemma.
The correct formulation of the aforementioned lemma should be as follows.

\textbf{Lemma 4.1.} \textit{Let the birth and death rates of a birth-and-death process be $\lambda_n$ and $\mu_n$ all belonging to $(0,\infty)$.
Then, the birth-and-death process is transient if there exist $c>1$ and a number $n_0$ such that for all $n>n_0$
\begin{equation}\label{1}
\frac{\lambda_n}{\mu_n}\geq1+\frac{1}{n}+\frac{c}{n\ln n},
\end{equation}
and is recurrent if there exists a number $n_0$ such that for all $n>n_0$
\begin{equation}\label{2}
\frac{\lambda_n}{\mu_n}\leq1+\frac{1}{n}+\frac{1}{n\ln n}.
\end{equation}
}

\begin{proof}
Following \cite{KM} p.370, a birth-and-death process is recurrent if and only if
$$
\sum_{n=1}^\infty\prod_{k=1}^{n}\frac{\mu_k}{\lambda_k}=\infty.
$$
Write
\begin{equation}\label{7}
\sum_{n=1}^\infty\prod_{k=1}^{n}\frac{\mu_k}{\lambda_k}=\sum_{n=1}^\infty\exp\left(\sum_{k=1}^{n}\ln\left(\frac{\mu_k}{\lambda_k}\right)\right).
\end{equation}
Now, suppose that \eqref{1} holds. Then, for sufficiently large $n$
$$
\frac{\mu_n}{\lambda_n}\leq1-\frac{1}{n}-\frac{c}{n\ln n}+O\left(\frac{1}{n^2}\right),
$$
and since the function $x\mapsto\ln x$ is increasing on $(0,\infty)$, then
$$
\begin{aligned}
\ln\left(\frac{\mu_n}{\lambda_n}\right)&\leq\ln\left(1-\frac{1}{n}-\frac{c}{n\ln n}+O\left(\frac{1}{n^2}\right)\right)\\
&=-\frac{1}{n}-\frac{c}{n\ln n}+O\left(\frac{1}{n^2}\right).
\end{aligned}
$$
Hence, for sufficiently large $n$
$$
\sum_{k=1}^n\ln\left(\frac{\mu_k}{\lambda_k}\right)\leq-\ln n-c\ln\ln n+O(1),
$$
and thus, by \eqref{7}, for some constant $C<\infty$,
$$
\sum_{n=1}^{\infty}\prod_{k=1}^{n}\frac{\mu_k}{\lambda_k}\leq C\sum_{n=1}^{\infty}\frac{1}{n(\ln n)^c}<\infty,
$$
provided that $c>1$. The transience follows.

On the other hand, suppose that \eqref{2} holds. Then, for sufficiently large $n$
$$
\frac{\mu_n}{\lambda_n}\geq1-\frac{1}{n}-\frac{1}{n\ln n}+O\left(\frac{1}{n^2}\right),
$$
and, consequently,
$$
\ln\left(\frac{\mu_n}{\lambda_n}\right)\geq\ln\left(1-\frac{1}{n}-\frac{1}{n\ln n}+O\left(\frac{1}{n^2}\right)\right).
$$
Similarly to that was provided before, for some constant $C^\prime$,
$$
\sum_{n=1}^{\infty}\prod_{k=1}^{n}\frac{\mu_k}{\lambda_k}\geq C^\prime\sum_{n=1}^{\infty}\frac{1}{n\ln n}=\infty.
$$
The recurrence follows.\end{proof}

As $n\to\infty$, asymptotic expansion (4.5) obtained in the proof of Lemma 4.2 in \cite{A} guarantees its correctness. However, the corrected version of Lemma 4.1 requires more delicate arguments in the proofs of Lemma 4.2 and Theorem 4.13 in \cite{A}. Specifically, in the proof of Lemma 4.2 instead of limit relation (4.6) we should study the cases $d=2$ and $d\geq3$ separately in terms of the present formulation of Lemma 4.1.

In the formulation of Theorem 4.13 in \cite{A}, assumption (4.12) must be replaced by the stronger one:
$$
\frac{L_n}{M_n}\leq 1+\frac{2-d}{n}+\frac{1-\epsilon}{n\ln n}
$$
for all large $n$ and small positive $\epsilon$. In the proof of Theorem 4.13 in \cite{A} we should take into account that for large $n$
$$
\frac{\lambda_n(1,d)}{\mu_n(1,d)}=1+\frac{d-1}{n}+O\left(\frac{1}{n^2}\right)
$$
is satisfied (see the proof of Lemma 4.2 in \cite{A}), and
hence,
$$
\frac{p_n}{1-p_n}\asymp\frac{\lambda_n(1,d)}{\mu_n(1,d)}\cdot\frac{L_n}{M_n}\leq1+\frac{1}{n}+\frac{1-\epsilon}{n\ln n}+\frac{K}{n^2},
$$
for a fixed constant $K$ and large $n$. So, according to Lemma 4.1 the process is recurrent.

Note that the statements of Lemma 4.1 are closely related to those of Theorem 3 in \cite{MAI} that prove recurrence and transience for the model studied there.

\medskip
\textbf{Typo correction:} The asymptotic relation at the end of page 1919 in \cite{A} should be
$$
\left(\frac{L_n}{M_n}\right)^{\lfloor\kappa n\rfloor}\left(\frac{L_n}{M_n}\right)^{n-\lfloor\kappa n\rfloor}\asymp\left(\frac{L_n}{M_n}\right)^n.
$$

\medskip
\textbf{Acknowledgement.} The help of the reviewer is highly appreciated.